\documentclass[12pt]{amsart}
\usepackage{amsmath,amssymb}
\usepackage{amsfonts}
\usepackage{amsthm}
\usepackage{latexsym}
\usepackage{graphicx}

\def\p{\partial}

\def\ii{\sqrt{-1}}

\def\R{\mathbb{R}}

\def\vv<#1>{\langle#1\rangle}
\def\sph{\mathbb{S}}

\def\XXint#1#2{\setbox0=\hbox{$#1{#2}{\int}$}{#2}\kern-.5\wd0 }

\def\XXint#1#2#3{{\setbox0=\hbox{$#1{#2#3}{\int}$}
     \vcenter{\hbox{$#2#3$}}\kern-.5\wd0}}

\def\vv<#1>{\langle#1\rangle}
\newtheorem{thm}{Theorem}[section]

\newtheorem{lem}{Lemma}[section]

\newtheorem{cor}{Corollary}[section]
\theoremstyle{definition}

\theoremstyle{remark}

\numberwithin{equation}{section}

\begin{document}

\title{Heat kernel recurrence on space forms and applications}

\author{Chengjie Yu$^1$}
\address{Department of Mathematics, Shantou University, Shantou, Guangdong, 515063, China}
\email{cjyu@stu.edu.cn}
\author{Feifei Zhao}
\address{Department of Mathematics, Shantou University, Shantou, Guangdong, 515063, China}
\email{14ffzhao@stu.edu.cn}
\thanks{$^1$Research partially supported by an NSF project of China with contract no. 11571215.}

\renewcommand{\subjclassname}{%
  \textup{2010} Mathematics Subject Classification}
\subjclass[2010]{Primary 35K05; Secondary 53C44}
\date{}
\keywords{Heat equation, Heat kernel, recurrence relation}
\begin{abstract}
In this paper, we first give a direct proof for two recurrence relations of the heat kernels for hyperbolic spaces in \cite{DM}. Then, by similar computation, we give two similar recurrence relations of the heat kernels for spheres. Finally, as an application, we compute the diagonal of heat kernels for odd dimensional hyperbolic spaces and the heat trace asymptotic expansions for odd dimensional spheres.
\end{abstract}
\maketitle\markboth{Yu \& Zhao}{Heat kernel recurrence}
\section{Introduction}
Let $K_n(t,r(x,y))$ be the heat kernel of the hyperbolic space $\mathbb H^n$. In \cite{DM}, the authors obtained the following two recurrence relations:
\begin{equation}\label{eq-h-re-s2}
K_{n+2}=-\frac{e^{-nt}}{2\pi\sinh r}\p_r K_n,
\end{equation}
and
\begin{equation}\label{eq-h-re-s1}
K_n(t,r)=\sqrt 2e^{\frac{(2n-1)t}{4}}\int_r^\infty \frac{K_{n+1}(t,\rho)\sinh\rho}{(\cosh\rho-\cosh r)^\frac12}d\rho,
\end{equation}
by the expression of $K_n$ computed by using Selberg's transform. Another method of obtaining the expressions of heat kernels on hyperbolic spaces using wave kernels can be found in \cite{GN}. The recurrence relations \eqref{eq-h-re-s2} and \eqref{eq-h-re-s1} are useful in obtaining heat kernel estimates on hyperbolic spaces. For examples, in \cite{DM}, the authors obtained sharp upper and lower bound of the heat kernels for hyperbolic spaces, and in \cite{YZ}, we obtained optimal Li-Yau gradient estimate for hyperbolic spaces by using \eqref{eq-h-re-s2} and \eqref{eq-h-re-s1}.

Let $H_n(t,r(x,y))$ be the heat kernel of $\R^n$. By the expression
\begin{equation}
H_n(t,r)=\frac{1}{(4\pi t)^\frac n2}e^{-\frac{r^2}{4t}}
\end{equation}
of $H_n$, it is not hard to see that two similar recurrence relations:
\begin{equation}\label{eq-e-re-s2}
H_{n+2}=-\frac{1}{2\pi r}\p_r H_n,
\end{equation}
and
\begin{equation}\label{eq-e-re-s1}
H_n(t,r)= 2\int_r^\infty \frac{H_{n+1}(t,\rho)\rho}{(\rho^2- r^2)^\frac12}d\rho
\end{equation}
are also true on $\R^n$. According to this, a natural question is: are there any similar recurrence relations on spheres? The same question was asked in \cite{Na} where the author only obtained asymptotic recurrence relations of the heat kernels on spheres for short distance. In this paper, we first give a direct proof of the recurrence relations \eqref{eq-h-re-s2} and \eqref{eq-h-re-s1} on hyperbolic spaces without using the expressions of the heat kernels. Then, by  similar computation, we give an affirmative answer to the question. More precisely, we obtain the following two recurrence relations of heat kernels on spheres.
\begin{thm}\label{thm-main}Let $\kappa_n(t,r(x,y))$ be the heat kernel of $\sph^n$ for $n=1,2,\cdots$. Then
\begin{enumerate}
\item \begin{equation}\label{eq-s-re-s2}
\kappa_{n+2}=-\frac{e^{nt}}{2\pi\sin r}\p_r\kappa_n
\end{equation}
for $n=1,2,\cdots$
\item
\begin{equation}\label{eq-s-re-s1}
\begin{split}
\kappa_n=&\sqrt 2e^{-\frac{2n-1}{4}t}\int_{r}^\pi\frac{\kappa_{n+1}(t,\rho)\sin\rho}{(\cos r-\cos \rho)^\frac12}d\rho\\
&+(n-1)2^{n-1}\omega_{n-2}\int_0^te^{-\frac{2n-1}{4}s}\kappa_{n+1}(s,\pi)\int_0^\pi\sin^{n-1}\left(\frac\rho 2\right)\cos^{n-2}\left(\frac\rho 2\right)\\
&\int_0^\pi\kappa_{n}(t-s,\arccos(\cos r\cos\rho+\sin r\sin\rho\cos\theta))\sin^{n-2}\theta d\theta d\rho ds
\end{split}
\end{equation}
for $n=2,3,\cdots$.
\end{enumerate}
\end{thm}
The reason that  \eqref{eq-s-re-s1} is not true for $n=1$ is mainly  that removability of singularity for  heat equations does not hold for $n=1$ (see \cite{Hs,Hu,NS,Sk}).

From the recurrence relation \eqref{eq-s-re-s2}, one can immediately obtain  explicit expressions of heat kernels for odd dimensional spheres:
\begin{equation}
\kappa_{2m+1}=\frac{e^{m^2 t}}{(2\pi)^m}\left(-\frac{1}{\sin r}\p_r\right)^{m}\kappa_1
\end{equation}
for $m=1,2,\cdots$, where
\begin{equation}
\kappa_1(t,r)=\frac{1}{\sqrt{4\pi t}}\sum_{k=-\infty}^\infty e^{-\frac{(r+2k\pi )^2}{4t}}.
\end{equation}
This expression was already known in literature, see for examples \cite{CT,Ho,Wu}. Moreover, by letting $r=0,\pi$ in \eqref{eq-s-re-s1}, we have the following recurrence relations:
\begin{equation}
\begin{split}
\kappa_n(t,0)=&2e^{-\frac{2n-1}{4}t}\int_{0}^\pi\kappa_{n+1}(t,\rho)\cos\left(\frac{\rho}{2}\right)d\rho+(n-1)2^{n-1}\omega_{n-1}\times\\
&\int_0^te^{-\frac{2n-1}{4}s}\kappa_{n+1}(s,\pi)\int_0^\pi\kappa_n(t-s,\rho)\sin^{n-1}\left(\frac\rho 2\right)\cos^{n-2}\left(\frac\rho 2\right) d\rho ds.
\end{split}
\end{equation}
and
\begin{equation}
\begin{split}
&\kappa_n(t,\pi)\\
=&(n-1)2^{n-1}\omega_{n-1}\int_0^te^{-\frac{2n-1}{4}s}\kappa_{n+1}(s,\pi)\int_0^\pi\kappa_n(t-s,\rho)\sin^{n-2}\left(\frac\rho 2\right)\cos^{n-1}\left(\frac\rho 2\right) d\rho ds.
\end{split}
\end{equation}

As an application of the recurrence relations \eqref{eq-h-re-s2} and \eqref{eq-s-re-s2}, we have the following recurrence relations of the diagonal of heat kernels and $\kappa_n(t,\pi)$.
\begin{cor} \label{cor-main}Let the notations be the same as before. Then, we have the following recurrence relations:

\begin{enumerate}
\item $$K_{n+2}(t,0)=-\frac{e^{-nt}}{2n\pi}\p_t K_n(t,0).$$
\item $$\kappa_{n+2}(t,0)=-\frac{e^{nt}}{2n\pi}\p_t \kappa_n(t,0).$$
\item $$\kappa_{n+2}(t,\pi)=\frac{e^{nt}}{2n\pi}\p_t\kappa_n(t,\pi).$$
\end{enumerate}
\end{cor}
The recurrence relations are useful in the  computation of the diagonal of heat kernels for hyperbolic spaces and spheres. As an application, we use them to compute the diagonal of heat kernels of odd dimensional hyperbolic spaces and the heat trace asymptotic expansions of odd dimensional spheres. Although heat trace coefficients of spheres have been obtained in \cite{CW,Ca,MS,Po,Wu}, the recurrence relation in (2) of Corollary \ref{cor-main}  provides a different and simpler approach.
\begin{thm}\label{thm-heat-trace}Let the notations be the same as before. Then,
\begin{enumerate}
\item \begin{equation}
K_{2m+1}(t,0)=(4\pi t)^{-\frac{2m+1}{2}}e^{-m^2t}\sum_{k=0}^{m-1}\frac{\Gamma(m-k+\frac12)c_{m,k}}{\Gamma(m+\frac12)}t^k
\end{equation}
for $m=1,2,\cdots$.
\item
\begin{equation}
\begin{split}
\kappa_{2m+1}(t,0)\sim&(4\pi t)^{-\frac{2m+1}{2}}e^{m^2t}\sum_{k=0}^{m-1}(-1)^k\frac{\Gamma(m-k+\frac12)c_{m,k}}{\Gamma(m+\frac12)}t^k\\
=&\frac{1}{{\rm Vol}(\sph^{2m+1})}t^{-\frac{2m+1}{2}}\sum_{k=0}^\infty a_{2m+1,k}t^k
\end{split}
\end{equation}
as $t\to0^+$, where
\begin{equation}\label{eq-heat-coef}
a_{2m+1,k}=\sum_{l=0}^k(-1)^l\frac{m^{2k-2l}\Gamma\left(m-l+\frac12\right)c_{m,l}}{(2m)!(k-l)!},\ \ k=0,1,2,\cdots
\end{equation}
are the heat trace coefficients of $\sph^{2m+1}$.
\end{enumerate}
Here
\begin{equation}
c_{m,k}=\left\{\begin{array}{ll}1&k=0\\
\sum_{1\leq i_1<i_2<\cdots<i_k\leq m-1}i_1^2i_2^2\cdots i_k^2&1\leq k\leq m-1\\
0&k\geq m
\end{array}\right.
\end{equation}
for $m=1,2,\cdots$.
\end{thm}
One should note that \eqref{eq-heat-coef} has been obtained in \cite[Theorem 1.3.1]{Po}. We would also like to mention that explicit expressions of heat kernels of some symmetric spaces were presented in \cite{AO,Ca}, and in \cite{FJ}, an expression of the heat kernel of $\sph^2$ in series was presented.

The organization of the rest of this paper is as follows: In Section 2, we give a direct proof to \eqref{eq-h-re-s2} and \eqref{eq-h-re-s1} and prove (1) of Corollary \ref{cor-main}. In Section 3, we prove Theorem \ref{thm-main} and (2) and (3) of Corollary \ref{cor-main}. Finally, in Section 4, we prove Theorem \ref{thm-heat-trace}.

\section{A direct proof of heat kernel recurrence on hyperbolic spaces}
In this section, we give proofs of the recurrence relations \eqref{eq-h-re-s2} and \eqref{eq-h-re-s1} for hyperbolic spaces by direct computation .

Let $h(t,r)$ be a smooth function. To check that $h(t,r(x,y))$ is the heat kernel of the $n$-dimensional hyperbolic space $\mathbb H^n$, we only need to  check that
\begin{equation}\label{eq-h-h-1}
h_t-\Delta_n h=h_t-h_{rr}-(n-1)\coth(r)h_r=0
\end{equation}
and
\begin{equation}\label{eq-h-h-2}
\int_{0}^\infty h(t,r)f(r)\omega_{n-1}\sinh^{n-1}rdr\to f(0)
\end{equation}
as $t\to0^+$ for any smooth function $f$ with compact support. Here and throughout this section, $\Delta_n$ is the Laplacian operator on $\mathbb H^n$ and  $$\omega_{n-1}=\frac{2\pi^\frac n2}{\Gamma(\frac n2)}$$
is the volume  of the $n-1$ dimensional sphere.
 Moreover, setting $\sigma=\cosh r$, then \eqref{eq-h-h-1}  is equivalent to
\begin{equation}\label{eq-h-h-1'}
\p_th-[(\sigma^2-1)\p_\sigma^2h+n\sigma\p_\sigma h]=0.
\end{equation}

We now give a direct proof to \eqref{eq-h-re-s2}.
\begin{thm}
Let $K_n(t,r(x,y))$ be the heat kernel of $\mathbb H^n$. Then,
\begin{equation}
K_{n+2}=-\frac{e^{-nt}}{2\pi\sinh r}\p_r K_n.
\end{equation}
\end{thm}
\begin{proof}Since $K_n(t,r(x,y))$ is a smooth function, it is not hard to see that $\p_rK_n(t,0)=0$ and $\frac{1}{\sinh r}\p_r K_n(t,r(x,y))$ is a smooth function (see \cite[Proposition 2.7]{KW}).

Moreover, note that
\begin{equation}
\p_tK_n-[(\sigma^2-1)\p_\sigma^2K_n+n\sigma\p_\sigma K_n]=0.
\end{equation}
Taking derivative on the last equality with respect to $\sigma$ gives us
\begin{equation}
  \p_t\p_\sigma K_n-[(\sigma^2-1)\p_\sigma^2\p_\sigma K_n+(n+2)\sigma\p_\sigma\p_\sigma K_n]-n\p_\sigma K_n=0
\end{equation}
So
\begin{equation}
(\p_t-\Delta_{n+2})[e^{-nt}\p_\sigma K_n]=0.
\end{equation}
Furthermore, for any smooth function $f(r)$ with compact support,
\begin{equation}
\begin{split}
&-\int_{\mathbb H^{n+2}}\frac{e^{-nt}}{2\pi\sinh r}\p_rK_n(t,r)f(r)dV\\
=&-\omega_{n+1}\int_0^\infty\frac{e^{-nt}}{2\pi}\p_rK_n(t,r)f(r)\sinh^{n}rdr\\
=&\frac{e^{-nt}\omega_{n+1}}{2\pi}\int_0^\infty K_n(t,r)(f(r)\sinh^nr)_rdr\\
=&\frac{e^{-nt}\omega_{n+1}}{2\pi\omega_{n-1}}\int_0^\infty K_n(f_r\sinh r+nf(r)\cosh r)\omega_{n-1}\sinh^{n-1}rdr\\
\to& f(0)
\end{split}
\end{equation}
as $t\to 0^+$, by noting that $K_n$ is the heat kernel of $\mathbb H^n$.  Then, by noting that $\p_\sigma K_n=\frac{1}{\sinh r}\p_r K_n$, we get the conclusion.
\end{proof}
Next, we come to prove (1) of Corollary \ref{cor-main}.
\begin{cor}\label{cor-re-h-trace}Let $K_n(t,r(x,y))$ be the heat kernel of $\mathbb H^n$. Then,
\begin{equation}
K_{n+2}(t,0)=-\frac{e^{-nt}}{2n\pi}\p_t K_n(t,0).
\end{equation}
\end{cor}
\begin{proof}
Setting $r=0$ in
\begin{equation}
\p_tK_n-[\p_r^2K_n+(n-1)\coth r\p_rK_n]=0,
\end{equation}
we have
\begin{equation}
\p_t K_n(t,0)=n\p_r^2K_n(t,0).
\end{equation}
Then, by \eqref{eq-h-re-s2},
\begin{equation}
  K_{n+2}(t,0)=-\frac{e^{-nt}}{2\pi}\p_r^2K_{n}(t,0)=-\frac{e^{-nt}}{2n\pi}\p_tK_{n}(t,0).
\end{equation}
\end{proof}
We next come to give a direct proof to  \eqref{eq-h-re-s1}.
\begin{thm}Let $K_n(t,r(x,y))$ be the heat kernel of $\mathbb H^n$. Then,
\begin{equation}
K_n(t,r)=\sqrt 2e^{\frac{(2n-1)t}{4}}\int_r^\infty \frac{K_{n+1}(t,\rho)\sinh\rho}{(\cosh\rho-\cosh r)^\frac12}d\rho.
\end{equation}
\end{thm}
\begin{proof}
Let $\sigma=\cosh r$, $s=\cosh\rho$ and $\xi=s-\sigma$. Then
\begin{equation*}
\begin{split}
\int_{r}^\infty\frac{K_{n+1}(t,\rho)\sinh \rho}{(\cosh \rho-\cosh r)^\frac12}d\rho=\int_\sigma^\infty\frac{K_{n+1}}{\sqrt{s-\sigma}}ds=\int_0^\infty\frac{K_{n+1}(t,\xi+\sigma)}{\sqrt \xi}d\xi.\\
\end{split}
\end{equation*}
Hence,
\begin{equation*}
\begin{split}
&(\p_t-\Delta_n)\int_{r}^\infty\frac{K_{n+1}(t,\rho)\sinh \rho}{(\cosh \rho-\cosh r)^\frac12}d\rho\\
=&\int_0^\infty\frac{(\p_t-((\sigma^2-1)\p_\sigma^2+n\sigma\p_\sigma))K_{n+1}(t,\xi+\sigma)}{\sqrt \xi}d\xi\\
=&\int_0^\infty\frac{[(((\sigma+\xi)^2-1)\p_\sigma^2+(n+1)(\sigma+\xi)\p_\sigma)-((\sigma^2-1)\p_\sigma^2+n\sigma\p_\sigma)]K_{n+1}(t,\xi+\sigma)}{\sqrt \xi}d\xi\\
=&\int_0^\infty\xi^{-\frac12}[(2\sigma\xi+\xi^2)\p_\xi^2+((n+1)\xi+\sigma)\p_\xi]K_{n+1}(t,\sigma+\xi)d\xi\\
=&-\int_0^\infty\left(\sigma\xi^{-\frac12}+\frac32\xi^\frac12\right)\p_\xi K_{n+1}(t,\sigma+\xi)d\xi+\int_0^\infty\xi^{-\frac12}[((n+1)\xi+\sigma)\p_\xi]K_{n+1}(t,\sigma+\xi)d\xi\\
=&\left(n-\frac12\right)\int_{0}^\infty\xi^\frac12 \p_\xi K_{n+1}(t,\sigma+\xi)d\xi\\
=&-\left(\frac{n}{2}-\frac{1}4\right)\int_0^\infty\xi^{-\frac12}K_{n+1}(t,\sigma+\xi)d\xi\\
=&-\left(\frac{n}{2}-\frac{1}4\right)\int_{r}^\infty\frac{K_{n+1}(t,\rho)\sinh \rho}{(\cosh \rho-\cosh r)^\frac12}d\rho.
\end{split}
\end{equation*}
So,
\begin{equation}
(\p_t-\Delta_n)\left[e^{\frac{2n-1}{4}t}\int_{r}^\infty\frac{K_{n+1}(t,\rho)\sinh \rho}{(\cosh \rho-\cosh r)^\frac12}d\rho\right]=0.
\end{equation}
Moreover,
\begin{equation}
\begin{split}
&\sqrt 2\omega_{n-1}e^{\frac{(2n-1)t}{4}}\int_0^\infty\int_r^\infty\frac{K_{n+1}(t,\rho)\sinh\rho}{(\cosh\rho-\cosh r)^\frac12}d\rho f(r)\sinh^{n-1}rdr\\
=&\sqrt 2\omega_{n-1}e^{\frac{(2n-1)t}{4}}\int_0^\infty K_{n+1}(t,\rho)\sinh\rho\int_0^\rho\frac{f(r)\sinh^{n-1}r}{(\cosh\rho-\cosh r)^\frac12}drd\rho\\
\to&\frac{\sqrt 2\omega_{n-1}}{\omega_n}\lim_{\rho\to 0}\sinh^{1-n}\rho\int_0^\rho\frac{f(r)\sinh^{n-1}r}{(\cosh\rho-\cosh r)^\frac12}dr\\
=&f(0)
\end{split}
\end{equation}
as $t\to 0^+$. The last equality can be computed as follows.

Let $x=\cosh r-1$, $y=\cosh \rho-1$ and $z=\frac xy$. Then,
\begin{equation}
\begin{split}
&\sinh^{1-n}\rho\int_0^\rho\frac{\sinh^{n-1}r}{(\cosh \rho-\cosh r)^\frac12}dr\\
=&[(1+y)^2-1]^{\frac{1-n}{2}}\int_0^{y}\frac{[(1+x)^2-1]^\frac{n-2}{2}}{(y-x)^\frac12}dx\\
=&(2+y)^\frac{1-n}2\int_0^1z^\frac{n-2}{2}(1-z)^{-\frac12}(2+yz)^\frac{n-2}2dz\\
\to&\frac{1}{\sqrt 2}B\left(\frac12,\frac n2\right)\\
=&\frac{\omega_{n}}{\sqrt 2 \omega_{n-1}}
\end{split}
\end{equation}
as $\rho\to 0^+.$ This completes the proof of theorem.
\end{proof}

\section{Heat kernel recurrence on spheres}
In this section, we come prove to the recurrence relations \eqref{eq-s-re-s2} and \eqref{eq-s-re-s1} for heat kernels on spheres. Similarly as before, to check that $h(t,r(x,y))$ is the heat kernel of $\sph^n$, we only need to  check that
\begin{equation}\label{eq-h-s-1}
h_t-\Delta_n h=h_t-h_{rr}-(n-1)\cot(r)h_r=0
\end{equation}
and
\begin{equation}\label{eq-h-s-2}
\int_{0}^\pi h(t,r)f(r)\omega_{n-1}\sin^{n-1}rdr\to f(0)
\end{equation}
as $t\to0^+$ for any smooth function $f(r)$. Here and throughout this section, $\Delta_n$ is the Laplacian operator on $\sph^n$. Moreover, by setting $\sigma=\cos r$, the equation \eqref{eq-h-s-1}  is equivalent to
\begin{equation}\label{eq-h-s-1'}
\p_th-[(1-\sigma^2)\p_\sigma^2h-n\sigma\p_\sigma h]=0.
\end{equation}
We first come to proof \eqref{eq-s-re-s2}.
\begin{thm}
Let $\kappa_n(t,r(x,y))$ be the heat kernel of $\sph^n$. Then
\begin{equation}
\kappa_{n+2}=-\frac{e^{nt}}{2\pi\sin r}\p_r\kappa_n.
\end{equation}
\end{thm}
\begin{proof} Since $\kappa_n(t,r(x,y))$ is smooth on $\sph^n$, one can see that $\p_r\kappa(t,0)=\p_r\kappa(t,\pi)=0$ and $\frac{1}{\sin r}\p_r\kappa_n(t,r(x,y))$ is a smooth function (see \cite[Proposition 2.7]{KW}).

Note that
\begin{equation}
\p_t\kappa_n-[(1-\sigma^2)\p_\sigma^2\kappa_n-n\sigma\p_\sigma\kappa_n]=0.
\end{equation}
Taking derivative on the last equality with respect to $\sigma$, we have
\begin{equation}
\p_t\p_\sigma\kappa_n-[(1-\sigma^2)\p_\sigma^2\p_\sigma \kappa_n-(n+2)\sigma\p_\sigma\p_\sigma\kappa_n]+n\p_\sigma\kappa_n=0.
\end{equation}
Thus,
\begin{equation}
(\p_t-\Delta_{n+2})[e^{nt}\p_\sigma \kappa_n]=0.
\end{equation}
Moreover,
\begin{equation}
\begin{split}
&-\int_{\mathbb S^{n+2}}\frac{e^{nt}}{2\pi\sin r}\p_r\kappa_n(t,r)f(r)dV\\
=&-\omega_{n+1}\int_0^\pi\frac{e^{nt}}{2\pi}\p_r\kappa_n(t,r)f(r)\sin^{n}rdr\\
=&\frac{e^{nt}\omega_{n+1}}{2\pi}\int_0^\pi \kappa_n(t,r)(f(r)\sin^nr)_rdr\\
=&\frac{e^{nt}\omega_{n+1}}{2\pi\omega_{n-1}}\int_0^\infty \kappa_n(f_r\sin r+nf(r)\cos r)\omega_{n-1}\sin^{n-1}rdr\\
\to& f(0)
\end{split}
\end{equation}
as $t\to 0^+$, where we have used that $\kappa_n$ is the heat kernel on $\sph^n$. Noting that $\p_\sigma \kappa_n=-\frac{1}{\sin r}\p_r \kappa_n$, we get the conclusion.
\end{proof}

Next, we come to prove (2) and (3) of Corollary \ref{cor-main}
\begin{cor}
Let $\kappa_n(t,r(x,y))$ be the heat kernel of $\sph^n$. Then,
\begin{enumerate}
\item $\kappa_{n+2}(t,0)=-\frac{e^{nt}}{2n\pi}\p_t\kappa_n(t,0)$.
\item $\kappa_{n+2}(t,\pi)=\frac{e^{nt}}{2n\pi}\p_t\kappa_n(t,\pi)$.
\end{enumerate}
\end{cor}
\begin{proof} Similarly as in the proof of Corollary \ref{cor-re-h-trace},
setting $r=0,\pi$ in
\begin{equation}
\p_t\kappa_n-[\p_r^2\kappa_n+(n-1)\cot r\p_r\kappa_n]=0
\end{equation}
and \eqref{eq-s-re-s2} will give us the conclusion.
\end{proof}
For the proof of  \eqref{eq-s-re-s1}, we need the following lemmas.
\begin{lem}\label{lem-sphere-1}
Let $\kappa_n(t,r(x,y))$ be the heat kernel of $\sph^n$. Then,
\begin{equation*}
\begin{split}
&(\p_t-\Delta_n)\int_{r}^\pi\frac{\kappa_{n+1}(t,\rho)\sin\rho}{(\cos r-\cos \rho)^\frac12}d\rho\\
=&\frac{2n-1}{4}\int_{r}^\pi\frac{\kappa_{n+1}(t,\rho)\sin\rho}{(\cos r-\cos \rho)^\frac12}d\rho-(n-1)(1+\cos r)^{-\frac12}\kappa_{n+1}(t,\pi).
\end{split}
\end{equation*}
\end{lem}
\begin{proof}
(1) Let $\sigma=\cos r$, $s=\cos\rho$ and $\xi=\sigma-s$. Then,
\begin{equation*}
\begin{split}
\int_{r}^\pi\frac{\kappa_{n+1}(t,\rho)\sin\rho}{(\cos r-\cos \rho)^\frac12}d\rho=\int_{-1}^\sigma\frac{\kappa_{n+1}(t,s)}{(\sigma-s)^\frac12}ds=\int_{0}^{1+\sigma}\frac{\kappa_{n+1}(t,\sigma-\xi)}{\xi^\frac12}d\xi.\\
\end{split}
\end{equation*}
Hence,
\begin{equation*}
\begin{split}
&(\p_t-\Delta_n)\int_{r}^\pi\frac{\kappa_{n+1}(t,\rho)\sin\rho}{(\cos r-\cos \rho)^\frac12}d\rho\\
=&\{\p_t-[(1-\sigma^2)\p_\sigma^2-n\sigma\p_\sigma]\}\int_{0}^{1+\sigma}\frac{\kappa_{n+1}(t,\sigma-\xi)}{\xi^\frac12}d\xi\\
=&\int_{0}^{1+\sigma}\frac{\p_t\kappa_{n+1}(t,\sigma-\xi)}{\xi^\frac12}d\xi+\left(\frac12+\left(n-\frac12\right)\sigma\right)(1+\sigma)^{-\frac12}\kappa_{n+1}(t,\pi)\\
&-(1-\sigma)(1+\sigma)^\frac12\p_\sigma\kappa_{n+1}(t,\pi)-\int_{0}^{1+\sigma}\frac{[(1-\sigma^2)\p_\sigma^2-n\sigma\p_\sigma]\kappa_{n+1}(t,\sigma-\xi)}{\xi^\frac12}d\xi\\
=&\int_{0}^{1+\sigma}\frac{[(1-(\sigma-\xi)^2)\p_\sigma^2-(n+1)(\sigma-\xi)\p_\sigma]\kappa_{n+1}(t,\sigma-\xi)}{\xi^\frac12}d\xi\\
&+\left(\frac12+\left(n-\frac12\right)\sigma\right)(1+\sigma)^{-\frac12}\kappa_{n+1}(t,\pi)-\int_{0}^{1+\sigma}\frac{[(1-\sigma^2)\p_\sigma^2-n\sigma\p_\sigma]\kappa_{n+1}(t,\sigma-\xi)}{\xi^\frac12}d\xi\\
&-(1-\sigma)(1+\sigma)^\frac12\p_\sigma\kappa_{n+1}(t,\pi)\\
=&\int_{0}^{1+\sigma}\frac{[(2\sigma\xi-\xi^2)\p_\sigma^2+((n+1)\xi-\sigma)\p_\sigma]\kappa_{n+1}(t,\sigma-\xi)}{\xi^\frac12}d\xi\\
&+\left(\frac12+\left(n-\frac12\right)\sigma\right)(1+\sigma)^{-\frac12}\kappa_{n+1}(t,\pi)-(1-\sigma)(1+\sigma)^\frac12\p_\sigma\kappa_{n+1}(t,\pi)\\
=&\int_0^{1+\sigma}\left(-\sigma\xi^{-\frac12}+\frac32\xi^\frac12\right)\p_\xi\kappa_{n+1}(t,\sigma-\xi)d\xi-\int_{0}^{1+\sigma}\xi^{-\frac12}((n+1)\xi-\sigma)\p_\xi\kappa_{n+1}(t,\sigma-\xi)d\xi\\
&+\left(\frac12+\left(n-\frac12\right)\sigma\right)(1+\sigma)^{-\frac12}\kappa_{n+1}(t,\pi)\\
=&-(n-\frac12)\int_0^{1+\sigma}\xi^\frac12\p_\xi\kappa_{n+1}(t,\sigma-\xi)d\xi+\left(\frac12+\left(n-\frac12\right)\sigma\right)(1+\sigma)^{-\frac12}\kappa_{n+1}(t,\pi)\\
=&\frac{2n-1}{4}\int_{r}^\pi\frac{\kappa_{n+1}(t,\rho)\sin\rho}{(\cos r-\cos \rho)^\frac12}d\rho-(n-1)(1+\sigma)^{-\frac12}\kappa_{n+1}(t,\pi).
\end{split}
\end{equation*}
\end{proof}
\begin{lem}\label{lem-sphere-2}
Let $\kappa_n(t,r(x,y))$ be the heat kernel of $\sph^n$. Then,
\begin{equation*}
\sqrt 2\omega_{n-1}\int_0^\pi\int_{r}^\pi \frac{\kappa_{n+1}(t,\rho)\sin\rho}{(\cos r-\cos \rho)^\frac12}d\rho f(r)\sin^{n-1}rdr\to f(0)
\end{equation*}
as $t\to 0^+$ for any smooth function $f$.
\end{lem}
\begin{proof}
Note that
\begin{equation}
\begin{split}
&\omega_{n-1}\int_0^\pi\int_{r}^\pi\frac{\kappa_{n+1}(t,\rho)\sin\rho}{(\cos r-\cos \rho)^\frac12}d\rho f(r)\sin^{n-1}rdr\\
=&\omega_{n-1}\int_0^\pi\kappa_{n+1}(t,\rho)\sin\rho\int_{0}^\rho\frac{f(r)\sin^{n-1}r}{(\cos r-\cos \rho)^\frac12}dr d\rho \\
=&\frac{\omega_{n-1}}{\omega_n}\lim_{\rho\to 0}\sin^{1-n}\rho\int_{0}^\rho\frac{f(r)\sin^{n-1}r}{(\cos r-\cos \rho)^\frac12}dr.
\end{split}
\end{equation}
Moreover, let $\sigma=1-\cos \rho$ and $s=1-\cos r$, $\xi=\frac{s}{\sigma}$, we have,
\begin{equation}
\begin{split}
&\sin^{1-n}\rho\int_{0}^\rho\frac{\sin^{n-1}r}{(\cos r-\cos \rho)^\frac12}dr\\
=&[1-(1-\sigma)^2]^\frac{1-n}{2}\int_0^\sigma\frac{[1-(1-s)^2]^\frac{n-2}{2}}{(\sigma-s)^\frac12}ds\\
=&[\sigma(2-\sigma)]^\frac{1-n}{2}\int_0^\sigma\frac{[s(2-s)]^\frac{n-2}{2}}{(\sigma-s)^\frac12}ds\\
=&(2-\sigma)^\frac{1-n}{2}\int_0^1\frac{\xi^\frac{n-2}{2}(2-\sigma\xi)^\frac{n-2}{2}}{(1-\xi)^\frac12}d\xi\\
\to&\frac{1}{\sqrt 2}B(1/2,n/2)\\
=&\frac{\omega_n}{\sqrt 2\omega_{n-1}}.\\
\end{split}
\end{equation}
This gives us the conclusion.

\end{proof}
We are now ready to prove \eqref{eq-s-re-s1}.
\begin{thm}
Let $k_n(t,r(x,y))$ be the heat kernel of $\sph^n$. Then
\begin{equation}
\begin{split}
\kappa_n=&\sqrt 2e^{-\frac{2n-1}{4}t}\int_{r}^\pi\frac{\kappa_{n+1}(t,\rho)\sin\rho}{(\cos r-\cos \rho)^\frac12}d\rho\\
&+(n-1)2^{n-1}\omega_{n-2}\int_0^te^{-\frac{2n-1}{4}s}\kappa_{n+1}(s,\pi)\int_0^\pi\sin^{n-1}\left(\frac\rho 2\right)\cos^{n-2}\left(\frac\rho 2\right)\\
&\int_0^\pi\kappa_{n}(t-s,\arccos(\cos r\cos\rho+\sin r\sin\rho\cos\theta))\sin^{n-2}\theta d\theta d\rho ds
\end{split}
\end{equation}
for $n=2,3,\cdots$
\end{thm}
\begin{proof}
Let $v$ be the solution of the Cauchy problem:
\begin{equation}
\left\{\begin{array}{l}v_t-\Delta_n v=(n-1)e^{-\frac{2n-1}{4}t}\kappa_{n+1}(t,\pi)\sec\left(\frac{r(o,x)}{2}\right)\\
v(0,x)=0
\end{array}\right.
\end{equation}
on $\sph^n$. Here $o$ is a fixed point in $\sph^n$. By Duhamel's principle,
\begin{equation}
\begin{split}
v(x)=(n-1)\int_0^te^{-\frac{2n-1}{4}s}\kappa_{n+1}(s,\pi)\int_{\sph^{n}}\kappa_n(t-s,r(x,y))\sec\left(\frac{r(o,y)}{2}\right)dV(y)ds.
\end{split}
\end{equation}
It is clear that $v$ is rotationally symmetric with respect to $o$. Moreover, by the spherical law of cosines, one has
\begin{equation}
\begin{split}
v(x)=&(n-1)\omega_{n-2}\int_0^te^{-\frac{2n-1}{4}s}\kappa_{n+1}(s,\pi)\int_0^\pi\sec\left(\frac{\rho}{2}\right)\sin^{n-1}\rho\\
&\int_0^\pi\kappa_{n}(t-s,\arccos(\cos r\cos\rho+\sin r\sin\rho\cos\theta))\sin^{n-2}\theta d\theta d\rho ds\\
=&(n-1)2^{n-1}\omega_{n-2}\int_0^te^{-\frac{2n-1}{4}s}\kappa_{n+1}(s,\pi)\int_0^\pi\sin^{n-1}\left(\frac{\rho}{2}\right)\cos^{n-2}\left(\frac{\rho}{2}\right)\\
&\int_0^\pi\kappa_{n}(t-s,\arccos(\cos r\cos\rho+\sin r\sin\rho\cos\theta))\sin^{n-2}\theta d\theta d\rho ds\\
\end{split}
\end{equation}
where $r=r(o,x)$. Moreover, by Lemma \ref{lem-sphere-1}, the function
\begin{equation}
u(t,x)=\sqrt 2e^{-\frac{2n-1}{4}t}\int_{r}^\pi\frac{\kappa_{n+1}(t,\rho)\sin\rho}{(\cos r-\cos \rho)^\frac12}d\rho+v(x)
\end{equation}
satisfies the heat equation on $\sph^n\setminus \{o'\}$ where $o'$ is the antipodal point of $o$. Note that $u$ is continuous, by removability of singularity of the heat equation (see \cite{Hs,Hu,NS,Sk}), we know that $u$ is smooth. Moreover, by Lemma \ref{lem-sphere-2}, $u(t,x)\to \delta_o$ as $t\to 0^+$. This completes the proof of the theorem.
\end{proof}

\section{Heat trace of odd dimensional hyperbolic spaces and spheres}
In this section, by using Corollary \ref{cor-main}, we prove Theorem \ref{thm-heat-trace}.

We first prove (1) of Theorem \ref{thm-heat-trace}.
\begin{proof}[Proof of (1) of Theorem \ref{thm-heat-trace}]
Suppose that
\begin{equation}
K_{2m+1}(t,0)=(4\pi t)^{-\frac{2m+1}{2}}e^{-m^2t}P_m(t).
\end{equation}
Then, by (1) of Corollary \ref{cor-main},
\begin{equation}
P_m=\left(1+\frac{2(m-1)^2}{2m-1}t\right)P_{m-1}-\frac{2t}{2m-1}P'_{m-1}
\end{equation}
with $P_0=1$. Let
\begin{equation}
Q_m=\Gamma\left(m+\frac12\right)P_m.
\end{equation}
Then
\begin{equation}\label{eq-Q-re}
Q_m=\left(m-\frac12+(m-1)^2t\right)Q_{m-1}-tQ'_{m-1}.
\end{equation}
Let
\begin{equation}
F_m(z)=\prod_{k=0}^{m-1}(1+k^2z)=\sum_{k=0}^{m-1}c_{m,k}z^k.
\end{equation}
We claim that
\begin{equation}
Q_m=\sum_{k=0}^{m-1}\frac{\Gamma(m-k+\frac12)t^k}{2\pi\ii}\int_{C}\frac{F_m(z)}{z^{k+1}}dz=\sum_{k=0}^{m-1}\Gamma\left(m-k+\frac12\right)c_{m,k}t^k,
\end{equation}
where $C$ is the unit circle. We will show this by induction. It is clearly true for $m=1$ and suppose it is true for $m-1$,  by \eqref{eq-Q-re},
\begin{equation}
\begin{split}
Q_m=&\left(m-\frac12+(m-1)^2t\right)Q_{m-1}-tQ'_{m-1}\\
=&\left(m-\frac12+(m-1)^2t\right)\sum_{k=0}^{m-2}\frac{\Gamma(m-k-\frac12)t^k}{2\pi\ii}\int_{C}\frac{F_{m-1}(z)}{z^{k+1}}dz\\
&-t\sum_{k=0}^{m-2}\frac{\Gamma(m-k-\frac12)kt^{k-1}}{2\pi\ii}\int_{C}\frac{F_{m-1}(z)}{z^{k+1}}dz\\
=&\sum_{k=0}^{m-1}\frac{t^k}{2\pi\ii}\int_{C}\left(\frac{(m-k-\frac12)\Gamma(m-k-\frac12)}{z^{k+1}}+\frac{(m-1)^2\Gamma(m-k+\frac12)}{z^k}\right)F_{m-1}(z)dz\\
=&\sum_{k=0}^{m-1}\frac{\Gamma(m-k+\frac12)t^k}{2\pi\ii}\int_{C}\frac{F_m(z)}{z^{k+1}}dz.
\end{split}
\end{equation}
So,
\begin{equation}
P_m=\sum_{k=0}^{m-1}\frac{\Gamma(m-k+\frac12)}{\Gamma(m+\frac12)}c_{m,k}t^k.
\end{equation}
This completes the proof.
\end{proof}

We next come to prove (2) of Theorem \ref{thm-heat-trace}. Note that
\begin{equation}\label{eq-s-k-1-0}
\kappa_1(t,0)=(4\pi t)^{-\frac12}\left(1+2\sum_{k=1}^\infty e^{-\frac{k^2\pi^2}{t}}\right).
\end{equation}
Because $t^{m}e^{-\frac{k^2\pi^2}{t}}$ tends to $0$ exponentially as $t\to 0^+$ for any constant $m$ and any positive constant $k$, by (2) of Corollary \ref{cor-main}, the terms $e^{-\frac{k^2\pi^2}{t}}$ with $k=1,2,\cdots$ in \eqref{eq-s-k-1-0} have no contribution to the heat trace asymptotic as $t\to 0^+$ for odd dimensional sphere. So, to compute the heat trace asymptotic for odd dimensional sphere, we only need to take care of $(4\pi t)^{-\frac12}$ in the expression \eqref{eq-s-k-1-0} of $\kappa_1(t,0)$.
\begin{proof}[Proof of (2) of Theorem \ref{thm-heat-trace}]

Let $\bar\kappa_1(t)=(4\pi t)^{-\frac 12}$ and
\begin{equation}\label{eq-re-b-k}
\bar\kappa_{n+2}=-\frac{e^{nt}}{2n\pi}\p_t\bar\kappa_n.
\end{equation}
Then, $\bar\kappa_{2m+1}$ is the heat trace asymptotic for $\sph^{2m+1}$. Suppose that
\begin{equation}
\bar\kappa_{2m+1}(t)=(4\pi t)^{-\frac{2m+1}{2}}e^{m^2t}p_m(t).
\end{equation}
Then, by \eqref{eq-re-b-k}, we know that
\begin{equation}
p_m=\left(1-\frac{2(m-1)^2}{2m-1}t\right)p_{m-1}-\frac{2t}{2m-1}p'_{m-1}
\end{equation}
with $p_0=1$. Let $q_m=\Gamma(m+\frac{1}2)p_m$. Then,
\begin{equation}\label{eq-q-re}
q_m=\left(m-\frac12-(m-1)^2t\right)q_{m-1}-tq'_{m-1}
\end{equation}
with $q_0=\Gamma(\frac12)$. Let
\begin{equation}
\tilde F_m(z)=\prod_{k=0}^{m-1}(1-k^2z)=\sum_{k=0}^{m-1}(-1)^kc_{m,k}z^k.
\end{equation}
We claim that
\begin{equation}
q_m=\sum_{k=0}^{m-1}\frac{\Gamma(m-k+\frac12)t^k}{2\pi\ii}\int_{C}\frac{\tilde F_m(z)}{z^{k+1}}dz=\sum_{k=0}^{m-1}(-1)^k\Gamma\left(m-k+\frac12\right) c_{m,k}t^k,
\end{equation}
where $C$ is the unit circle. We will show this by induction. It clearly true for $m=1$ and suppose it is true for $m-1$, by \eqref{eq-q-re},
\begin{equation}
\begin{split}
q_m=&\left(m-\frac12-(m-1)^2t\right)q_{m-1}-tq'_{m-1}\\
=&\left(m-\frac12-(m-1)^2t\right)\sum_{k=0}^{m-2}\frac{\Gamma(m-k-\frac12)t^k}{2\pi\ii}\int_{C}\frac{\tilde F_{m-1}(z)}{z^{k+1}}dz\\
&-t\sum_{k=0}^{m-2}\frac{\Gamma(m-k-\frac12)kt^{k-1}}{2\pi\ii}\int_{C}\frac{\tilde F_{m-1}(z)}{z^{k+1}}dz\\
=&\sum_{k=0}^{m-1}\frac{t^k}{2\pi\ii}\int_{C}\left(\frac{(m-k-\frac12)\Gamma(m-k-\frac12)}{z^{k+1}}-\frac{(m-1)^2\Gamma(m-k+\frac12)}{z^k}\right)\tilde F_{m-1}(z)dz\\
=&\sum_{k=0}^{m-1}\frac{\Gamma(m-k+\frac12)t^k}{2\pi\ii}\int_{C}\frac{\tilde F_m(z)}{z^{k+1}}dz.
\end{split}
\end{equation}
So,
\begin{equation}
p_m=\sum_{k=0}^{m-1}(-1)^k\frac{\Gamma\left(m-k+\frac12\right)}{\Gamma(m+\frac12)}c_{m,k}t^k.
\end{equation}
and
\begin{equation}
\begin{split}
\bar \kappa_{2m+1}=&(4\pi t)^{-\frac{2m+1}2}e^{m^2t}p_m\\
=&(4\pi t)^{-\frac{2m+1}2}\sum_{k=0}^\infty \frac{m^{2k}}{k!}t^k\sum_{k=0}^{m-1}(-1)^k\frac{\Gamma\left(m-k+\frac12\right)}{\Gamma(m+\frac12)}c_{m,k}t^k\\
=&(4\pi t)^{-\frac{2m+1}2}\sum_{k=0}^{\infty}\sum_{l=0}^{k}(-1)^l\frac{m^{2k-2l}\Gamma\left(m-l+\frac12\right)}{\Gamma(m+\frac12)(k-l)!}c_{m,l}t^k
\end{split}
\end{equation}
So the heat coefficients of $\sph^{2m+1}$ are
\begin{equation}
\begin{split}
a_k=&\omega_{2m+1}(4\pi)^{-\frac{2m+1}2}\sum_{l=0}^{k}(-1)^l\frac{m^{2k-2l}\Gamma\left(m-l+\frac12\right)}{\Gamma(m+\frac12)(k-l)!}c_{m,l}\\
=&\sum_{l=0}^k(-1)^l\frac{m^{2k-2l}\Gamma\left(m-l+\frac12\right)}{(2m)!(k-l)!}c_{m,l}
\end{split}
\end{equation}
for $k=0,1,\cdots$. This completes the proof.
\end{proof}

\end{document}